\documentclass[12pt]{amsart}

\usepackage{graphicx}
\usepackage{amssymb}
\usepackage{amsthm}
\usepackage{amsfonts}
\usepackage{amsmath}
\usepackage{amstext}

\usepackage{graphicx}
\usepackage{amssymb}
\usepackage{amsthm}
\usepackage{amsfonts}
\usepackage{amsmath}
\usepackage{amstext}
\usepackage{mathrsfs}

\theoremstyle{plain}
\newtheorem{thm}{Theorem}[section]

\newtheorem{lem}[thm]{Lemma} 
\newtheorem{prop}[thm]{Proposition}

\theoremstyle{definition}

\theoremstyle{remark}

\numberwithin{equation}{section}

\newcommand{\R}{\mathbb{R}}

\newcommand{\Cn}{\mathscr {C}}
\newcommand{\AL}{\mathscr {S}}

\DeclareMathOperator{\grad}{\rm {grad }}
\DeclareMathOperator{\dist}{\rm {dist }}

\title[Regular covers for open relatively compact subanalytic sets]{Regular covers of open relatively compact subanalytic sets}

\author[A. Parusi\' nski]{Adam Parusi\'nski }

\address {Laboratoire J.A. Dieudonn\'e UMR CNRS 7351, 
Universit\'e de Nice - Sophia Antipolis,
Parc Valrose,
06108 Nice Cedex 02,
France}
   
\email{adam.parusinski@unice.fr}

\begin{document}



\maketitle 
Let $M$ be a real analytic manifold of dimension $n$.  In this paper we study the algebra $\AL (M)$ of  relatively compact  open subanalytic subsets of $M$.  As we show this algebra is  generated by sets with Lipschitz regular boundaries. More precisely, we call a relatively compact  open  subanalytic subset $U\subset M$ \emph{an open  subanalytic Lipschitz ball} if its closure is subanalytically bi-Lipschitz homeomorphic to the unit ball of $\R^n$.  
Here we assume that $M$ is equipped with a Riemannian metric.  
Any two such metrics are equivalent on relatively compact sets and hence the above definition 
is independent of the choice of a metric.

\begin{thm}\label{algebra}
The algebra $\AL (M)$ is generated by open subanalytic Lipschitz balls.
\end{thm}

That is to say  if $U$  is a relatively compact open subanalytic subset
of $M$ then  the characteristic function $ 1_U$ is a linear combination of functions of the form
$1 _{W_1}, ... ,  1 _{W_m}$, where the $W_j$ are open subanalytic Lipschitz balls.  
Note that, in general, $U$ cannot be covered by subanalytic Lipschitz balls, as it is easy to see 
for $\{(x,y)\in \R^2; y^2<x^3, x<1\}$, $M=\R^2$, due to the presence of cusps.  Nevertheless we show the 
existence of a "regular" cover in the sense that we control the distance to the boundary.

\begin{thm}\label{forschapira}
Let $U\in \AL (M)$.  Then there exist a finite cover  $U= \bigcup _i U_i$ by open 
subanalytic sets such that :
\begin{enumerate}
\item
every $U_i$ is subanalytically homeomorphic to an open $n$-dimensional ball;
\item
there is $C>0$ such that for every  $ x\in U$, $\dist (x, M\setminus U)\leq C \max_i \dist (x,M \setminus U_i)$ 
\end{enumerate}
\end{thm}

 The proofs 
of Theorems \ref{algebra} and \ref{forschapira} 
are based on the regular projection theorem, cf. \cite{mostowski}, \cite{parusinski1988}, \cite{parusinski1994},  the 
classical cylindrical decomposition, and the L-regular decomposition of subanalytic sets, cf. \cite{kurdyka}, 
\cite{parusinski1994}, \cite{pawlucki}.  L-regular sets  are natural multidimensional generalization of classical cusps.  We recall them briefly in Subsection \ref{Lregular}.    We show also the following strengthening of Theorem \ref{forschapira}.

\begin{thm}\label{extra}
In Theorem \ref{forschapira} we may require additionally that all $U_i$ are open L-regular cells (i.e. interiors of $L$-regular sets).  
\end{thm}

For an open $U\subset M$ we denote $\partial U = \overline U \setminus U$.  


\section{Proofs}

\subsection{Reduction  to the case $M=\R^n$.}  Let $U\in \AL (M)$.   Choose a finite covering 
 $\overline U \subset \bigcup_i V_ i$  by open relatively compact sets such that for each $V_i$ 
 there is an open neighborhood of $\overline V_i$  analytically  diffeomorphic to $\R^n$.   Then there 
 are finitely many open subanalytic $U_{ij}$ such that $U_{ij}\subset V_i$ and $1_U$ is a combination 
 of $1_{U_{ij}}$.   Thus it suffices to show Theorem \ref{algebra} for relatively compact open subanalytic subsets of 
 $R^n$.  
 
 Similarly, it suffices to show Theorems \ref{forschapira} and \ref{extra} for $M=\R^n$.  Indeed, it follow from the observation that the function 
 $$
 x \to  \max_i \dist (x,M \setminus V_i)
 $$
 is continuous  and nowhere zero on $\bigcup_i V_ i$ and hence bounded from below by a nonzero constant $c>0$ 
 on $\overline U$.  Then 
 \begin{align*}
 \dist (x, M\setminus U)  \le & C_1 \le c^{-1} C_1  \max_i \dist (x,M \setminus V_i) 
 \end{align*}
 where  $C_1$ is the diameter of $\overline U$ and hence, if $c^{-1} C_1\ge 1$, 
  \begin{align*}
 \dist (x, M\setminus U)  \le  c^{-1} C_1  \max_i( \min \{  \dist (x, M\setminus U) , \dist (x,M \setminus V_i)\}.  
 \end{align*}
   Now if for each $U\cap V_i$ we choose a covering $U_{ij}$ satisfying the statement of Theorem 
 \ref{forschapira} or \ref{extra} then for $x\in U$ 
 \begin{align*}
 \dist (x, M\setminus U) &  \le c^{-1} C_1  \max_i( \min \{  \dist (x, M\setminus U) , \dist (x,M \setminus V_i)\} \\
 &  \le c^{-1} C_1   \max _i \dist (x,M \setminus U \cap V_i) \leq  Cc^{-1} C_1 \max_{ij} \dist (x,M \setminus U_{ij})
 \end{align*}
Thus the cover $U_{ij}$ satisfies the claim of Theorem 
 \ref{forschapira}, resp. Theorem \ref{extra}.   

\subsection{Regular projections } 
We recall after  \cite{parusinski1988}, \cite{parusinski1994} the subanalytic version of the regular projection theorem of T. Mostowski introduced originally in 
\cite{mostowski} for complex analytic sets germs.  

Let $X\subset \R^n$ be subanalytic.  For $\xi \in \R^{n-1}$ we  denote by   $\pi_\xi: \R^n\to \R^{n-1}$ the linear projection parallel to 
$(\xi,1) \in \R^{n-1} \times \R$.  Fix constants $C, \varepsilon >0$.   We say that $\pi = \pi_\xi$ is \emph{($C, \varepsilon $)-regular at $x_0\in \R^n$ (with respect to $X$)} if 
\begin{enumerate}
\item[(a)]
$\pi|_X$ is finite;
\item[(b)]
the intersection of $X$ with the open cone 
\begin{align}\label{cone}
\Cn_\varepsilon (x_0, \xi) = \{ x_0 + \lambda (\eta,1); |\eta -\xi |<\varepsilon, \lambda \in \R \setminus 0\}
\end{align}
is empty or a finite disjoint union of sets of the form 
$$
 \{ x_0 + \lambda_i (\eta) (\eta,1); |\eta -\xi |<\varepsilon\} ,
$$
where $\lambda_i$ are real analytic nowhere vanishing functions defined on $ |\eta -\xi |<\varepsilon$.  
\item [(c)]
the functions $\lambda_i$ from (b) satisfy for all $ |\eta -\xi |<\varepsilon$
$$
\| \grad \lambda_i (\eta) \| \le C |\lambda_i (\eta)|,
$$
\end{enumerate}
We say that $\mathcal P \subset \R^{n-1}$ defines  \emph{a set of regular projections for $X$} if there exists 
$C, \varepsilon >0$ such that  for every $x_0\in \R^n$ there is $\xi \in \mathcal P$ such that 
$\pi_\xi$ is $(C, \varepsilon )$-regular at $x_0$.   

\begin{thm}[\cite{parusinski1988}, \cite{parusinski1994}]
Let $X$ be a compact subanalytic subset of $\R^n$ such that $\dim X<n$. Then the generic set  of $n+1$ vectors 
$\xi_1, \ldots, \xi_{n+1}$, $\xi _i\in \R^{n-1}$, 
defines a set of regular projections for $X$.  \\
(Here by generic we mean in the complement of a subanalytic nowhere  dense subset of $(\R^{n-1})^{n+1}$.)
\end{thm}

\subsection{Cylindrical decomposition }  
We recall the first step of a basic construction, the   cylindrical algebraic decomposition, for details see for instance \cite{coste}, \cite{vandendries}. 

Set $X=\overline U \setminus U$.  Then $X$ is a compact subanalytic subset of $\R^n$ of dimension 
$n-1$.  We denote by $Z\subset X$ the set of singular points of $X$ that is the complement in $X$ of the set 
$$
Reg(X) : = \{x\in X; (X,x) \text{ is the germ of a real analytic submanifold of dimension } n-1\}. 
$$
Then $Z$ is closed in $X$, subanalytic and $\dim Z\le n-2$.  

Assume that  the standard projection $\pi : \R^n \to \R^{n-1}$ restricted to $X$ is finite.  Denote by 
$\Delta_\pi\subset \R^{n-1}$ the union of $\pi (Z)$ and the set of critical values of $\pi|_{Reg(X)}$.   Then $\Delta_\pi$, called \emph{the 
discriminant set of $\pi$}, is compact and subanalytic.  It is clear that $\overline {\pi(U)}= \pi(U) \cup \Delta_\pi$.  

\begin{prop} \label{cylinders}
Let $U'\subset \pi(U) \setminus \Delta_\pi$ be open and connected.   Then there are finitely many bounded real analytic functions 
$\varphi_1 <\varphi _2< \cdots < \varphi_k $ defined on $U'$, such that $X\cap \pi^{-1}(U')$ is the union of graphs of $\varphi_i$'s.  In particular,  $U\cap \pi^{-1}(U')$ is the union of the sets 
$$
\{(x',x_n) \in \R^n; x'\in U',  \varphi_i (x') < x_n <  \varphi_{i+1} (x'),\}
$$
and moreover, if $U'$ is subanalytically homeomorphic to an open $(n-1)$-dimensional ball, then each of these sets is 
subanalytically homeomorphic to an open $n$-dimensional ball.  
\end{prop}

\subsection{The case of a regular projection}
Fix $x_0\in U$ and suppose that $\pi : \R^n \to \R^{n-1}$ is $(C, \varepsilon )$-regular at $x_0\in \R^n$ with respect to $X$.  
Then the cone \eqref{cone} contains no point of $Z$.  By  \cite{parusinski1994} Lemma 5.2, this cone contains no critical point 
of $\pi|_{Reg(X)}$, provided $\varepsilon$ is chosen sufficiently small (for fixed $C$).  In particular, $x'_0=\pi (x_0) \not\in \Delta_\pi$.  

In what follows we fix $C,\varepsilon >0$ and suppose $\varepsilon$ small.  We denote the cone  \eqref{cone}  by $\Cn$ for short.  Then for $\tilde C$ sufficiently large, that depends only on $C$ 
and $\varepsilon$,  we have 
\begin{align}\label{bound2}
 \dist (x_0, X\setminus \Cn ) \le \tilde C  \dist (x'_0, \pi (X\setminus \Cn )) \le \tilde C \dist (x'_0, \Delta_\pi) .
\end{align}
The first inequality is obvious, the second follow from the fact that the singular part of $X$ and the 
critical points of $\pi|_{Reg(X)}$ are both outside the cone.  


\subsection{Proof of Theorem \ref{forschapira}} 
Induction on $n$.  Set $X=\overline U \setminus U$ and let $\pi_{\xi_1}, \ldots , \pi_{\xi_{n+1}}$ be 
a set of $(C,\varepsilon)$-regular projections with respect to $X$.  To each of these projections we apply the cylindrical decomposition.  
More precisely, let us fix one of these projections that for simplicity we suppose standard and denote it by $\pi$.  
Then we 
apply the inductive assumption to $\pi (U) \setminus \Delta_\pi$.  Thus  let  
$\pi (U) \setminus \Delta_\pi= \bigcup U'_i$ be a finite cover satisfying the statement of 
Theorem \ref{forschapira}.   Applying  to each 
$U'_i$  Proposition \ref{cylinders} we obtain a family of cylinders that cover $U\setminus \pi^{-1} (\Delta_\pi)$.  
In particular they cover the set of those points of $U$ at which $\pi$ is $(C,\varepsilon)$-regular.

\begin{lem}\label{lem_induction}
Suppose $\pi$ is $(C,\varepsilon)$-regular at $x_0\in U$.  Let $U'$ be an open subanalytic subset of $\pi(U)\setminus \Delta_\pi$ such that $x_0' = \pi(x_0)\in U'$ and 
\begin{align}\label{indassumption}
\dist (x'_0, \Delta_\pi)\le \tilde C  \dist (x'_0, \partial U'),
\end{align}
with $\tilde C\ge 1$ for which \eqref{bound2} holds.  Then 
\begin{align}\label{main}
\dist (x_0, X)\le (\tilde C)^2  \dist (x_0, \partial U_1),
\end{align}
where $U_1$ is the member of cylindrical decomposition of $U\cap \pi^{-1} (U')$ containing $x_0$.  
\end{lem}

\begin{proof}
We decompose  $\partial U_1$ into two parts.  The first one is vertical, i.e. contained in 
$\pi ^ {-1} (\partial U')$ and the second part is contained in $X$.  The distance to the first one from $x_0$ equals to  the horizontal distance, 
that is  $\dist (x'_0, \partial U')$.  Thus we have 
\begin{align}\label{distance2}
  \dist (x_0, \partial U_1) = \min \{ \dist (x_0, X),  \dist (x'_0, \partial U')\} .  
\end{align}
If $  \dist (x_0, \partial U_1)= \dist (x_0, X)$ then  \eqref{main} holds with $\tilde C=1$, otherwise 
   $ \dist (x_0, \partial U_1)  =  \dist (x'_0, \partial U') \le  \dist (x_0, X)$ and then by \eqref{indassumption} and \eqref{bound2}  
\begin{align}\label{bound4}
 \dist (x_0, X\setminus \Cn ) \le  \tilde C \dist (x'_0, \Delta_\pi)   \le (\tilde C)^2 \dist (x'_0, \partial U') \le 
   (\tilde C)^2   \dist (x_0, X)  .  
\end{align}
\end{proof}

Thus to complete the proof of Theorem  \ref{forschapira} it suffices to show that  the assumptions of 
Lemma \ref{lem_induction} are satisfied.  This follows from the inclusion $\partial \pi (U) \subset \Delta_\pi$ 
that gives $\dist (x'_0, \Delta_\pi)  \le \dist (x'_0, \partial  \pi (U))$, and from $\dist (x'_0, \partial \pi (U))\le \tilde C  \dist (x'_0, \partial U')$ that holds 
by the inductive assumption.  This ends the proof of Theorem \ref{forschapira}.

\subsection{L-regular sets}\label{Lregular}

Let $Y\subset \R^n$ be subanalytic, $\dim Y=n$. 
Then $Y$ is called \emph {L-regular (with respect to given system of coordinates)} if 
\begin{enumerate}
\item
if $n=1$ then $Y$ is a non-empty  closed bounded interval;
\item
if $n>1$ then $Y$ is of the form 
\begin{align}\label{L-regular}
Y = \{(x',x_n)\in \R^n;  f(x') \le x_n \le g(x'), x'\in Y' \},
\end{align}
where $Y'\subset \R^{n-1}$ is L-regular, $f$ and $g$ are continuous subanalytic functions defined in $Y'$.  It is also assumed that 
on the interior of $Y$, $f$ and $g$  are 
analytic, satisfy $f<g$, and have bounded first order 
partial derivatives.  
\end{enumerate}
If $\dim Y = k < n$ then we say that $Y$ is  \emph {L-regular (with respect to given system of coordinates)} if 
\begin{equation}\label{Lgraph}
Y=\{(y,z) \in \R^{k}\times \R^{n-k}; \, z = h(y), \, y \in Y' \} , 
\end{equation}
where $Y'\subset \R^{k}$ is L-regular, $\dim Y'=k$, $h$ is a continuous subanalytic map defined on $Y'$, 
such that $h$ is real analytic on the interior of $Y$, and has the  first order 
partial derivatives bounded.  

We say that $Y$ is \emph{L-regular} if it is L-regular with respect to a linear (or equivalently orthogonal) 
system of coordinates on $\R^n$.

We say that $A\subset \R^n$ is an  \emph{L-regular cell} if  $A$ is the relative interior of an L-regular set.  
That is, it is the interior of an L-regular set if $\dim A =n$,  and it is the graph of $h$ restricted to $Int(Y')$ 
for an  L-regular set of the form \eqref{Lgraph}.  By convention, every point is a zero-dimensional L-regular cell.   

By \cite{kurdyka}, see also Lemma 2.2 of \cite{parusinski1994} and Lemma 1.1 of \cite{kurdykaparusinski}, 
 L-regular sets and L-regular cells satisfy the following property, 
called in \cite{kurdyka} quasi-convexity.  We say that $Z\subset \R^n$ is \emph{quasi-convex} if  there is a 
constant $C>0$ such that every two points $x,y$ of $Z$
 can be connected in $Z$ by a continuous subanalytic 
arc of length bounded by $C\|x-y\|$.   It can be shown that for an L-regular set or cell $Y$ the constant 
$C$ depends only on $n$ and the bounds on first order partial derivatives  of functions describing $Y$ 
in the above definition.   
By  Lemma 2.2 of \cite{parusinski1994}, an L-regular cell is homeomorphic to the (open) unit ball.

Let $Y$ be a subanalytic subcet of a real analytic manifold $M$.  We say that $Y$ is \emph{$L$-regular} if there exists its neighborhood  $V$ in $M$ and an analytic diffeomorphism $\varphi :V \to \R^n$ such that 
$\varphi (Y)$ is $L$-regular.    Similarly we define an $L$-regular cell in $M$.  

\subsection{Proof of Theorem \ref{extra}}

Fix a constant $C_1$ sufficiently large and a  projection $\pi : \R^n \to \R^{n-1}$ that is assumed, for simplicity,
 to be the standard one.  
We suppose that $\pi$ restricted to $X=\partial U$ is finite.   We say that $x' \in \pi (U) \setminus \Delta_\pi$ is
 \emph{$C_1$-regularly covered} if there is a neighborhood $\tilde U'$ of $x'$ in  
$\pi (U) \setminus \Delta_\pi$ such that $X\cap \pi^{-1} (\tilde U')$ is the union of graphs of analytic functions 
with all first order partial 
derivatives bounded (in the absolute value) by $C_1$.  Denote by $U'(C_1)$ the set of all 
$x' \in \pi (U) \setminus \Delta_\pi$ that 
are $C_1$ regularly covered.   Then $U'(C_1)$ is open (if we use strict inequalities while defining it) 
and subanalytic.  
By Lemma 5.2 of  \cite{parusinski1994}, if $\pi$ is a $(C,\varepsilon)$-regular projection at $x_0$  then $x'_0$ is $C_1$-regularly covered, for $C_1$ sufficiently big $C_1\ge 
C_1(C,\varepsilon)$.  Moreover we have the following result.  

\begin{lem}\label{lem} 
Given positive constants $C,\varepsilon$.  Suppose that the constants $\tilde C$ and $C_1$  are chosen sufficiently big, 
$C_1\ge C_1(C,\varepsilon )$, $\tilde C\ge \tilde C(C,\varepsilon )$ .   
Let  $\pi$ be $(C,\varepsilon)$-regular  at $x_0\notin X$ and let 
\begin{align*}
V' = \{x'\in \R^{n-1}; \dist (x',x'_0) < (\tilde C)^{-1} dist (x_0, X\cap \Cn )\} .  
\end{align*}
Then  $\pi ^{-1} (V')\cap X\cap \Cn $ 
is the union of graphs of $\varphi_i$ with all first order partial derivatives bounded (in the absolute value) by $C_1$.  
Moreover, then either  $\pi ^{-1} (V')\cap (X\setminus \Cn) =  \emptyset $  or  
$$\dist (x'_0, \Delta_\pi) =  \dist (x'_0, \pi (X\setminus \Cn )) \le \dist (x'_0, \partial U'(C_1)) . $$ 
\end{lem}

\begin{proof} 
We only prove the second part of the statement since the first part follows from Lemma 5.2 of \cite{parusinski1994}.  
If $\pi ^{-1} (V')\cap X\setminus \Cn \ne   \emptyset $ then any point of $\pi (X\setminus \Cn )$ realizing 
$\dist (x'_0, \pi (X\setminus \Cn ))$ must be in the discriminant set $ \Delta_\pi$.  
\end{proof}

We now apply to $U'(C_1)$ the inductive hypothesis and thus assume that  $U'(C_1)= \bigcup U'_i$ is a 
finite regular cover by open $L$-regular cells.  Fix one of them $U'$ and let $U_1$ be a member of the cylindrical decomposition  of $U\cap \pi^{-1} (U')$.  Then $U_1$ is an L-regular cell.   Let $x_0\in U_1$.  We apply to $x_0$ Lemma \ref{lem}.  

If  $\pi ^{-1} (V')\cap (X\setminus \Cn) =  \emptyset $  then 
\begin{align*}
\dist (x_0, X)\le dist (x_0, X\cap \Cn )\le \tilde C  \dist (x'_0, \partial U'(C_1)) \le \tilde C^2  \dist (x'_0, \partial U'), 
\end{align*}
where the second inequality follows from the first part of Lemma \ref{lem} and the last inequality by the induction hypothesis.  
Then $\dist (x_0, X)\le  \tilde C^2  \dist (x_0, \partial U_1)$ follows from \eqref{distance2}.  

Otherwise, $\dist (x'_0, \Delta_\pi)  \le \dist (x'_0, \partial U'(C_1)) \le \tilde C  \dist (x'_0, \partial U')$ and the claim 
follows from Lemma \ref{lem_induction}.  This ends the proof.

\subsection{Proof of Theorem \ref{algebra}}
The proof is based on the following result.

\begin{thm}\label{dokladny} {\rm [Theorem A of \cite{kurdyka}]}
Let $Z_i \subset \R^n$ be a finite family of subanalytic sets.  Then there is   be a finite disjoint collection 
$\{A_j\}$ of L-regular cells such that each $Z_i$ is the disjoint union of some of $A_j$.  
\end{thm}

Similar results in  the (more general) o-minimal set-up are proven in \cite{kurdykaparusinski} and 
\cite{pawlucki}.

Let $U$ be a relatively compact open subanalytic subset of $\R^n$.  By Theorem \ref{dokladny}, $U$ 
is a disjoint union of  L-regular cells and hence it suffices to show the statement of Theorem \ref{algebra} for 
a relatively compact, not necessarily open,  L-regular cell.   We consider first the case of an open L-regular cell.  Thus suppose that 
\begin{align}\label{L-regular2}
U = \{(x',x_n)\in \R^n;  f(x') < x_n < g(x'), x'\in U' \},
\end{align}
where $U'$ is a relatively compact L-regular cell, $f$ and $g$ are subanalytic and 
analytic functions on $U'$ with the first 
order partial derivatives bounded.  Then, by the quasi-convexity of $U'$,  $f$ and $g$ are Lipschitz.   
 Thanks to the classical result of Banach, cf. \cite{banach} (7.5) p. 122, we may suppose  
that $f$ and $g$ are restrictions to $U'$ of Lipschitz subanalytic functions, denoted also by $f$ and $g$,  defined everywhere on $\R^{n-1}$ and satisfying $f\le g$.  
Indeed, Banach gives the following formula for such an extension  of a Lipschitz function $f$ defined on a subset 
of a metric space 
$$
\tilde f (p) = \sup_{q\in B} f(q) - L \|p-q\|,
$$
where $L$ is the Lipschitz constant of $f$.  Then $\tilde f$ is Lipschitz with the same constant as $f$ 
and subanalytic if so 
was $f$.  
By  the inductive assumption on dimension we may assume that $U$ is given by \eqref{L-regular2} with $U'$ an 
L-regular cell.   Denote 
$U$ by $U_{f,g}$ to stress its dependence on $f$ and $g$ (with $U'$ fixed).  
Then 
$$
1_{U_{f,g}} = 1_{U_{f-1,g}}+ 1_{U_{f,g+1}}- 1_{U_{f-1,g+1}} 
$$
and $U_{f-1,g}$,  $U_{f,g+1}$. and $U_{f-1,g+1}$ are open  subanalytic Lipschitz balls.  

Suppose now that 
\begin{equation}\label{Lgraph2}
U=\{(y,z) \in \R^{k}\times \R^{n-k}; \, z = h(y), \, y \in U' \} , 
\end{equation}
where $U'$ is a relatively compact open L-regular cell of $\R^k$, $h$ is a subanalytic and  analytic map defined  on $U'$ 
with the first 
order partial derivatives bounded.  Hence $h$ is Lipschitz.   We may again assume that $h$ is the restriction of a 
Lipschitz subanalytic map $h:\R^k \to \R^{n-k}$ and then, by the inductive hypothesis, that $U'$ is 
a subanalytic Lipschitz ball. 
Let 
$$
U_\emptyset =  \{ (y,z) \in U'\times \R^{n-k}; \,   h_i (y)-1 <z_i< h_i (y) +1\,  , i=1, ...,n-k\} 
$$
For $I\subset \{1, ..., n-k\}$ we denote 
$$
U_I =  \{ (y,z) \in U_\emptyset ;  \,   z_i\ne  h_i (y) \text { for }  i\in I\} .  
$$
Note that each $U_I$ is the disjoint union of $2^{|I|}$ of open subanalytic Lipschitz balls and that
$$
1_{U} =   
\sum_{I\subset \{1, ..., n-k\}} (-1) ^{|I|} 1_{U_I}.  
$$

This ends the proof.


\end{document}